\documentclass[bezier]{cpc}
\usepackage{ifthen,latexsym,amssymb,amsmath,graphics}

\usepackage{color}		
\usepackage{epsfig}

\newcommand{\C}[1]{{\protect\cal #1}}

\newcommand{\I}[1]{{\mathbb #1}}
\renewcommand{\O}[1]{\overline{#1}}
\newcommand{\ceil}[1]{\lceil #1\rceil}
\newcommand{\e}{\varepsilon}
\newcommand{\floor}[1]{\lfloor #1\rfloor}
\newcommand{\me}{{\mathrm e}}
\renewcommand{\mid}{:}

\newcommand{\beq}[1]{\begin{equation}\label{eq:#1}}
\newcommand{\eeq}{\end{equation}}
\newcommand{\req}[1]{\textrm{(\ref{eq:#1})}}

\newcommand{\bth}[2][nothing]{\ifthenelse{\equal{#1}{nothing}}
 {\begin{theorem}} {\begin{theorem}[#1]}\label{th:#2}}

\newcommand{\blm}[2][nothing]{\ifthenelse{\equal{#1}{nothing}}
 {\begin{lemma}} {\begin{lemma}[#1]}\label{lm:#2}}

\newcommand{\claim}[1]{\medskip\noindent{\bf Claim #1} }
\newcommand{\cqed}{\nolinebreak\mbox{\hspace{5 true pt}%
  \rule[-0.85 true pt]{2.0 true pt}{8.1 true pt}}}
\newcommand{\bcpf}{\medskip\noindent\textbf{Proof of Claim.} }
\newcommand{\ecpf}{\cqed \medskip}

\newcommand{\brm}{\smallskip\noindent{\bf Remark.} }

\begin{document}

\newcommand{\f}[1]{{\normalsize$#1$}} 
\newcommand{\NP}{NP}
\newcommand{\ELin}{{\sc E3-Lin-2}}
\newcommand{\m}{\mbox{\sc Max-Cut}}
\newcommand{\sat}{\mbox{\sc E3-SSAT}}

\title[Finding an An Unknown Acyclic Orientation]{Finding an Unknown Acyclic Orientation of a Given 
Graph\footnote{Reverts to public domain 28 years from publication.}
 }
\author[Oleg Pikhurko]{\spreadout{OLEG PIKHURKO}%
 \thanks{Partially supported by  the
National Science Foundation,  Grants 
DMS-0457512 and DMS-0758057.}\\
 \affilskip Department of Mathematical Sciences\\
 \affilskip Carnegie Mellon University\\
 \affilskip Pittsburgh, PA 15213, USA\\
Web: {\tt http://www.math.cmu.edu/\symbol{126}pikhurko}}

\maketitle

\begin{abstract}
 Let $c(G)$ be the smallest number of edges we have to test in order to
determine an unknown acyclic orientation of the given graph $G$ in the worst
case. For example, if $G$ is the complete graph on $n$ vertices, then $c(G)$
is the smallest number of comparisons needed to sort $n$ numbers.

We prove that $c(G)\le (1/4+o(1))\,n^2$ for any graph $G$ on $n$ vertices,
answering in the affirmative a question of Aigner, Triesch, and Tuza
[\emph{Discrete Mathematics}, \textbf{144} (1995) 3--10]. Also, we show that,
for every $\e>0$, it is \NP-hard to approximate the parameter $c(G)$ within a
multiplicative factor $74/73-\e$.\end{abstract}

\section{Introduction}

The \emph{acyclic orientation game} is the following. There are two players,
\emph{Algy} and \emph{Strategist}, to whom we shall also refer as \emph{him}
and \emph{her} correspondingly. Let $G$ be a given graph, known to both
players. At each step of the game, Algy selects any edge of $G$ and Strategist
has to orient this edge. The only restriction on Strategist's replies is that
the revealed orientation has to be \emph{acyclic}, that is, it does not
contain directed cycles. The game ends when the current partial orientation
extends to a \emph{unique} acyclic orientation of the whole graph $G$. Algy
tries to minimize the number of steps while Strategist aims at the
opposite. Let $c(G)$ be the length of the game assuming that both players play
optimally.

In other words, Algy wants to discover a `hidden' acyclic orientation of $G$
by querying edges. The parameter $c(G)$ measures the worst-case complexity,
that is, it is the smallest number such that Algy has a strategy that needs at
most $c(G)$ steps for every acyclic orientation of $G$.

The special case when $G=K_n$ (the complete graph on $n$ vertices) is
equivalent to the well-known \emph{minimum-comparison sorting} problem. While
the asymptotic result $c(K_n)=(1+o(1))\, n\log_2n$ of Ford and
Johnson~\cite{ford+johnson:59} is not hard to prove, the exact computation of
$c(K_n)$ seems very difficult. For example, the problem of computing
$c(K_{13})$ that appeared in Knuth's book~\cite[Chapter~5.3.1,
Exercise~35]{knuth:acp3} was solved only some 30 years later by
Peczarski~\cite{peczarski:02} (see also~\cite{peczarski:04}).

One interpretation of $c(G)$ for a general order-$n$ graph $G$ is that Algy
has to discover as much as possible information about the relative order of
$n$ elements given that certain pairs (namely, those corresponding to the
edges of the complementary graph $\O G$) cannot be queried. Manber and
Tompa~\cite{manber+tompa:81,manber+tompa:84} considered a related but
different problem where the player can query \emph{any} of the ${n\choose 2}$
possible pairs but has to find the relative order for every edge of the given
graph $G$.

Various results and bounds on $c(G)$ for general graphs were obtained by
Aigner, Triesch, and Tuza~\cite{aigner+triesch+tuza:95}, who in particular
studied graphs with $c(G)=e(G)$, calling them \emph{exhaustive}. Even this
property seems out of grasp. For example, the computational complexity of
checking whether $c(G)=e(G)$ is not known, see
Tuza~\cite[Problem~58]{tuza:01}. Alon and Tuza~\cite{alon+tuza:95} studied
$c(G)$, where $G\in\C G_{n,p}$ is a random graph of order $n$ with edge
probability $p$. They obtained, among other results, the correct order of
magnitude when $p$ is a non-zero constant: in this case $c(G)=\Theta(n\log n)$
almost surely.

The parameter $c(G)$ is not monotone with respect to the subgraph relation.
For example, while $c(K_n)=(1+o(1))n\log_2 n$, there are graphs $G$ of order
$n$ with $c(G)\ge\floor{n^2/4}$. Indeed, let $G$ be obtained from the
\emph{Tur\'an graph} $T_2(n)$, the complete bipartite graph with vertex parts
$V_1$ and $V_2$ of size $\floor{n/2}$ and $\ceil{n/2}$, by adding an arbitrary
bipartite graph $H$ inside one part of $T_2(n)$. Let $V(H)=U_1\cup U_2$ be a
bipartition of $H$. Suppose, for example, that $U_1,U_2\subseteq V_1$.
Strategist, in her replies, orients all edges from $U_1$ to $V(G)\setminus
U_1$ and from $V_2$ to $V_1\setminus U_1$. It is easy to see that Algy has to
ask about the orientation of \emph{every} edge of the original Tur\'an graph
$T_2(n)$, giving the claimed bound $c(G)\ge \floor{n/2}\, \ceil{n/2}=
\floor{n^2/4}$. We did not see any improvement over this bound in the
literature; our Proposition~\ref{pp:BeatTuran} improves it by $1$.

Aigner, Triesch, and Tuza~\cite[Page~10]{aigner+triesch+tuza:95} asked whether
the above bound is asymptotically sharp, that is, whether $c(G)\le
(1/4+o(1))\, n^2$ for every graph $G$ of order $n$. This open question is also
mentioned by Alon and Tuza~\cite[Page~263]{alon+tuza:95} and by
Tuza~\cite[Problem~55]{tuza:01}. Here we answer it in the affirmative.

\bth{1/4} For every $\e>0$ there is $n_0$ such that $c(G)\le (1/4+\e)n^2$ for
every graph $G$ of order $n\ge n_0$.\end{theorem}

Aigner, Triesch, and Tuza~\cite[Page~10]{aigner+triesch+tuza:95} also asked
if, furthermore, the upper bound can be improved to $n^2/4+C$ for some
absolute constant $C$. Unfortunately, we cannot prove this
strengthening.

Our proof of Theorem~\ref{th:1/4} shows more. Namely, for every $\e>0$ there
is $C$ such that, for any graph $G$ of order $n$, Algy can point (in one go) a
set $D$ of at most $Cn^{3/2} (\ln n)^{1/2}$ edges so that every acyclic
orientation of $D$ implies the orientation of all but at most $(1/4+\e)n^2$
remaining edges of $G$. Strategies of this type (when Algy has to send his
questions in a few \emph{rounds}) are useful
 in situations where the main limitation is on the
number of times that the players can exchange (large amounts of)
information. The study of comparison sorting in rounds was initiated by
Valiant~\cite{valiant:75}. We refer the reader to a survey by Gasarch, Golub,
and Kruskal~\cite{gasarch+golub+kruskal:03} for more information on the topic.

Aigner, Triesch, and Tuza~\cite[Page~10]{aigner+triesch+tuza:95} also asked
about the computational complexity of deciding whether $c(G)\le k$ on the
input $(G,k)$, see also Tuza~\cite[Problem~59]{tuza:01}. We obtain some
progress on this question as follows.

\bth{hardness} For every $\e>0$ it is \NP-hard to approximate $c(G)$ within a
multiplicative factor $74/73-\e$.\end{theorem}

It is possible that the acyclic orientation game is {PSPACE}-complete but the
author could not show this.

Also, one would like to complement Theorem~\ref{th:hardness} by providing a
polynomial time algorithm that approximates $c(G)$ within a multiplicative
factor $O(1)$. Unfortunately, the best approximability ratio in terms of
$n=v(G)$ that the author could find is $O(n/\log n)$: output $e(G)$ as an
upper bound on $c(G)$ and $e(G)\log_2 n/(Cn)$ as a lower bound, where $C$ is
the constant given by Theorem~\ref{th:lower} of Section~\ref{lower}. It is a
remaining open problem to close this gap.

\section{Notation}

We will use the standard graph terminology that can be found, for example, in
the books by Bollob\'as~\cite{bollobas:mgt} or Diestel~\cite{diestel:gt}. Some
of the less common conventions are as follows.

For brevity, we usually abbreviate an unordered pair $\{x,y\}$ as $xy$. We
write $(x,y)$ to denote that an edge $xy$ is oriented from $x$ to $y$. Let
$[n]=\{1,\dots,n\}$.

For a graph $G$ and disjoint sets of vertices $X,Y\subseteq V(G)$, $G[X,Y]$
denotes the bipartite graph on $X\cup Y$ consisting of all edges of $G$
connecting $X$ to $Y$. A \emph{cut} of $G$ is a partition $V(G)=V_1\cup
V_2$. Its \emph{value} is $e(G[V_1,V_2])$, the number of edges connecting
$V_1$ to $V_2$. If the graph $G$ comes equipped with the \emph{edge-weight}
function $w: E(G)\to\I R$, then the \emph{value} of the cut $\{V_1,V_2\}$ is
$\sum_{x_1\in V_1}\sum_{x_2\in V_2} w(x_1x_2)$. The \emph{max-cut} parameter
$\m(G)$ is the maximum value of a cut of $G$.

A partial order $\prec$ on $V(G)$ and an acyclic orientation of $E(G)$ are
\emph{compatible} if, for every edge $xy\in E(G)$, the elements $x$ and $y$
are comparable in the $\prec$-ordering and, moreover, $(x,y)$ if and only if
$x\prec y$. In this case, the phrases and expressions `$(x,y)$', `$y$ is above
$x$', `$x$ is smaller than $y$', `$y\succ x$', and so on, are all synonymous.

\section{Bounding $c(G)$ for Order-$n$ Graphs}\label{large}

\begin{proposition}\label{pp:BeatTuran} For every $n\ge 3$ there is a graph $G$ of order $n$ with
$c(G)\ge \floor{n^2/4}+1$.\end{proposition}
 \begin{proof} 
Let $G$ be the complete 3-parite graph with parts $X\cup Y\cup Z$ where
$|X|=|Y|=\floor{(n-1)/2}$. (Thus, depending on the parity, $n=2k+1$ or
$n=2k$, the part sizes are either $(k,k,1)$ or $(k-1,k-1,2)$.)

Strategist orients $(x,y)$ for every $x\in X$ and $y\in Y$ and answers Algy's
questions about these edges accordingly.

For every $z\in
Z$, Strategist does the following. She waits until Algy queries an edge
incident to $z$ for the first time. If this is an edge $xz$ with $x\in X$,
then Strategist orients all edges from $X\cup Y$ to $z$ (and answers all
Algy's questions accordingly). Note that Algy has to query every edge $yz$
with $y\in Y$ because neither of its orientations would create a directed
cycle in Strategist's ordering. Thus, Algy has to query at least $|Y|+1$ edges
at $z$ (including the first edge $xz$). Likewise, if the first queried edge
was $yz$ with $y\in Y$, then Strategist orients all edges from $z$ to $X\cup
Y$ and Algy has to query all edges $xz$ with $x\in X$. 

Also, independently of the game scenario on the edges adjacent to $Z$, Algy has
to query all edges between $X$ and $Y$. Thus $c(G)\ge |X|\times
|Y|+|Z|(|X|+1)$, which is easily seen to be the
required bound.\end{proof}

Next, we prove Theorem~\ref{th:1/4}. Its proof, where Algy queries random
edges in the first round, is somewhat similar to the methods of Bollob\'as and
Rosenfeld~\cite{bollobas+rosenfeld:81} (see also H\"aggkvist and
Hell~\cite{haggkvist+hell:81}) who studied how much information about the
unknown linear order can be obtained in just one round with the given number
of queries.

In order to prove Theorem~\ref{th:1/4} we will need the following auxiliary
result.

\bth[Ruzsa and Szemer\'edi~\cite{ruzsa+szemeredi:78}]{efr}
 For every $\e>0$ there is $\delta>0$ such that if a graph $G$ of order $n$
has at most $\delta n^3$ triangles then we can remove at most $\e n^2$ edges
from $G$, making it triangle-free.\end{theorem}

\begin{proof}[Proof of Theorem~\ref{th:1/4}.] Given $\e>0$, let
$\delta=\delta(\e/2)>0$ be the constant returned by Theorem~\ref{th:efr} on the
input $\e/2$. Fix an arbitrary positive constant $C$ such that
$C^2>2(1+\delta)/\delta^2$. Let $n$ be sufficiently large. Let $G$ be an
arbitrary graph of order $n$. Let $V=V(G)$, $p=C \sqrt{\ln n/n}$, and $w=\floor{\delta n/2}$.

Algy selects a set $D$ of edges of $G$ by including each element of $E(G)$
into $D$ with probability $p$, independently of the other choices. Let the
acronym \emph{whp} stand for `with high probability', meaning with probability
$1-o(1)$ as $n\to\infty$.

\claim1 Whp, the following holds for every linear ordering $L=(V,\prec)$ of
$V$ and every $2w$ pairwise distinct vertices $x_1,\dots,x_w,y_1,\dots,y_w\in
V$ with $x_i\prec y_i$ for $i\in[w]$. For $i\in [w]$, define
 \beq{Zi}
 Z_i=\{z\in V\mid x_i\prec z\prec y_i,\  x_iz,zy_i\in E(G)\}\setminus
\{x_1,\dots,x_{i-1},y_1,\dots,y_{i-1}\}.
 \eeq
 If each $Z_i$ has at least $\delta n$ elements, then there are $i\in[w]$ and
$z\in Z_i$ such that $x_iz$ and $zy_i$ belong to $D$.\medskip

\bcpf Fix any linear order $\prec$ on $V$ and $2w$ arbitrary pairwise distinct
vertices
$x_1,\dots,x_w,y_1,\dots,y_w$ such that $x_i\prec y_i$ and $|Z_i|\ge \delta n$
for each $i\in [w]$. Clearly, there are at most $n!\, n^{2w}$ choices of a such
configuration.

The probability that this configuration violates the claim is at most
$(1-p^2)^{\delta n w}$ because there are at least $w\times \delta n$ choices of
$(i,z)$ with $i\in[w]$ and $z\in Z_i$, the probability that at least one of
the edges $x_iz$ and $zy_i$ of $G$ is not in $D$ is $1-p^2$, while these
probabilities are independent over distinct pairs $(i,z)$. (Indeed, the events
for different pairs $(i,z)$ involve disjoint sets of edges; this was the
reason for excluding any vertex in $\{x_1,\dots,x_{i-1},y_1,\dots,y_{i-1}\}$
from $Z_i$ in~\req{Zi}.)

The union bound shows that the total probability of failure is at most
 $$
 n!\, n^{2w}(1-p^2)^{\delta n w}< \me^{n\ln n+2w\ln n - p^2\delta n w} \le
 \me^{(1+\delta - C^2 \delta^2/2+o(1))\,n\ln n}.
 $$
 This is $o(1)$ by the choice of $C$. The claim is proved.\ecpf

Also, whp $|D|\le pn^2$ 
by the Chernoff bound~\cite{chernoff:52}. Hence,
there is a set $D$ that satisfies the conclusion of Claim~1 and has at most
$pn^2$ elements. Fix such a set $D$.

During the first round, Algy asks about the orientation of all edges in
$D$. After we have received Strategist's answers, let $H$ be the spanning
subgraph of $G$ that consists of those edges of $G$ whose orientation is still
undetermined from the revealed orientation of $D$.

We claim that $H$ has at most $\delta n^3$ triangles. Suppose on the contrary
that this is false. Fix an arbitrary linear ordering $\prec$ of $V$ that is
compatible with the orientation of $D$. Let us define $x_i$ and $y_i$
inductively on $i$. Suppose that $i\in[w]$ and we have already defined
$x_1,\dots,x_{i-1}$ and $y_1,\dots,y_{i-1}$.

Let $U=\{x_1,\dots,x_{i-1},y_1,\dots,y_{i-1}\}$. The vertices in
$U$ belong to at most $2(i-1) {n\choose 2}< w n^2$ triangles of $H$. So,
the graph $H'=H-U$ has at least $\delta n^3- w n^2\ge \delta n^3/2$
triangles. By averaging, $H'$ contains a pair of vertices $x_{i}$ and $y_{i}$
such that there are at least $(\delta n^3/2)/{n\choose 2}> \delta n$ vertices
$z\in V(H')=V\setminus U$ such that $x_{i}\prec z\prec y_{i}$ and
$\{x_i,y_i,z\}$ spans a triangle in $H'$. Now, increase $i$ by 1 and iterate
the above step if the new index $i$ is still at most $w$.

For $i\in [w]$, let $Z_i$ 
be defined by
\req{Zi}; we have $|Z_i|\ge \delta n$.  
The obtained vertices $x_1,\dots,x_w,y_1,\dots,y_w$ satisfy all assumptions of
Claim~1 with respect to the linear order $\prec$.
By the definition of $D$ (which was chosen to satisfy the conclusion of
Claim~1), there are $i\in[w]$ and $z\in Z_i$ such that $x_iz,zy_i\in D$. By
the definition of $Z_i$, we have $x_i\prec z\prec y_i$. Since $\prec$ was
chosen to be compatible with Strategist's replies, the edges $x_iz,y_iz\in D$
are oriented as $(x_i,z)$ and $(z,y_i)$. Note that $x_i$ and $y_i$ are
adjacent in $H$ because these two vertices belong to at least $\delta n\ge 1$
triangles of $H$ by the definition of $x_i$ and $y_i$. But then the
orientation of the edge $x_iy_i\in E(G)$ is determined after the first round,
contradicting the fact that $x_iy_i\in E(H)$. Thus the graph $H$ of order $n$
has at most $\delta n^3$ triangles.

By the choice of $\delta$ (that is, by Theorem~\ref{th:efr}) there is a set
$F$ of at most $\e n^2/2$ edges such that $E(H)\setminus F$ contains no
triangles. By the Tur\'an theorem~\cite{turan:41} (or rather the special case
which was earlier proved by Mantel~\cite{mantel:07}) we have $|E(H)\setminus
F|\le n^2/4$. Thus $e(H)\le n^2/4+\e n^2/2$.

In the second round, Algy asks about the orientation of all edges of $H$. By
the definition of $H$, this completely determines the orientation of all edges
of $G$. Assuming that Strategist plays optimally, we have $$
 c(G)\le |D| + e(H)\le pn^2+\left(\frac{n^2}4 + \frac{\e n^2}2\right) \le\frac{n^2}4 +\e n^2, 
 $$
 finishing the proof of Theorem~\ref{th:1/4}.\end{proof}

\brm All known proofs of Theorem~\ref{th:efr} use some version of the
Regularity Lemma of Szemer\'edi~\cite{szemeredi:76} and therefore return a function
$\delta(\e)$ that approaches $0$ extremely slowly and is of little practical
value. Tao Jiang~\cite{jiang:pc} observed that instead of
Theorem~\ref{th:efr} one can use the result of Moon and
Moser~\cite{moon+moser:62} that a graph of order $n$ and size $m$ contains at
least $(m/3n)(4m-n^2)$ triangles. His calculations~\cite{jiang:pc} based on
this idea show that $c(G)\le n^2/4+2n^{7/4}(\ln n)^{1/4}$
for any order-$n$ graph $G$ with $n$ large. On the other hand, if we use 
Theorem~\ref{th:efr}, then we can deduce some structural information about
almost extremal graphs. Namely,
if an order-$n$ graph $G$ satisfies 
$c(G)=(\frac14+o(1))n^2$, then  by the Stability Theorem of Erd\H
os~\cite{erdos:67a} 
and Simonovits~\cite{simonovits:68} applied to the triangle-free graph
$H\setminus F$, there is a partition $V(G)=V_1\cup V_2$ with 
$(\frac14+o(1))n^2$ edges going across (which is asymptotically largest possible).
 Unfortunately, neither of these two approaches
has led us to the complete answer so far.\medskip

\section{Inapproximability Results}\label{hardness}

In order to prove Theorem~\ref{th:hardness} we will need the following
auxiliary result.

\blm{74} For every $\delta>0$, it is \NP-hard to approximate the graph parameter
 $$
 3e(G)+\m(G)
 $$
 within a multiplicative factor $74/73-\delta$.\end{lemma}

\begin{proof} We will use the construction of H\aa stad~\cite{hastad:97,hastad:01} that
demonstrates that {\sc Max-Cut} is \NP-hard to approximate within a factor
$17/16-\e$. Since we are interested in a somewhat different parameter than
just {\sc Max-Cut}, we have to unfold H\aa stad's construction.

First, H\aa stad proves~\cite[Theorem~2.3]{hastad:97} that it is \NP-hard to
approximate \ELin\ within factor less than 
$2$. That is, for every $\e>0$
it is \NP-hard to distinguish, for an input system $\C S$ of $s$
equations over $\I Z_2$ each of the form $x+y+z=0$ or $x+y+z=1$, 
between the cases when some assignment of
variables satisfies at least $(1-\e)s$ equations and when every assignment
satisfies at most $(1/2+\e)s$ equations.

Next, H\aa stad constructs~\cite[Theorem~4.2]{hastad:97} a graph $G$ from a
given instance $\C S$ of \ELin\ with $s$ equations as follows. We can assume
that $s_0\ge s/2$ equations of $\C S$ are of the form $x+y+z=0$. (If $s_0<
s/2$, we can simply replace each variable $x$ by $1-x$.) Let $s_1=s-s_0$ be
the number of equations of the form $x+y+z=1$.

Each equation $x+y+z=0$ and $x+y+z=1$ is replaced respectively by the
so-called \emph{$8$-gadget} and \emph{$9$-gadget} of Trevisan, Sorkin,
Sudan, and
Williams~\cite{trevisan+sorkin+sudan+williamson:96,trevisan+sorkin+sudan+williamson:00}.
The definition of these gadgets can be found in the journal
version~\cite[Lemmas~4.2 and~4.3]{trevisan+sorkin+sudan+williamson:00}. For
our purposes we need to know only that, for $\alpha=8$ or $9$, this particular
$\alpha$-gadget is an edge-weighted graph of total edge-weight
$\alpha+1$ whose vertex set consists of the variables $x$, $y$, and $z$, the
constant $0$, and some new vertices so that:
 \begin{itemize}
 \item every $0/1$-assignment of $x$, $y$, and $z$ that satisfies the equation
can be extended to a cut of value at least $\alpha$ but not to a cut of
a strictly larger value;
 \item no $0/1$-assignment of $x$, $y$, and $z$ that violates the equation can
be extended to a cut of value strictly larger than $\alpha-1$.
 \end{itemize}
 (Here, a cut in a gadget $H$ is encoded by an assignment $f:V(H)\to\{0,1\}$
with $f(0)=0$.)
Also, the special vertices (the variables and the constant 0) form an
independent set in both gadgets. Thus the constructed graph $G$ has total
edge-weight $9s_0+10s_1$. 

The above properties imply that if we can satisfy at least $(1-\e)s$ equations
of $\C S$ then $G$ has a cut of value at least $8s_0+9s_1-10\e s$. Also, if
every assignment of variables violates at least $(1/2-\e)s$ equations, then no
cut of $G$ can have value larger than $8s_0+9s_1-(1/2-\e)s$. Thus, if we cannot
distinguish these two alternatives for \ELin\ in polynomial time, then we
cannot distinguish in polynomial time whether, for edge-weighted graphs,
$3e(G)+\m(G)$ is at least $u_1$ or at most $u_2$, where
 \begin{eqnarray*}
 u_1&=&3(9s_0+10s_1)+8s_0+9s_1-10\e s\ =\ -4s_0+39s-10\e s,\\
 u_2&=&3(9s_0+10s_1)+8s_0+9s_1-(1/2-\e)s\ =\ -4s_0 + 38.5 s + \e s.
 \end{eqnarray*}
 When $\e<1/22$, then the ratio $u_1/u_2$ is minimized when $s_0=s/2$ is as
small as possible. Thus $u_1/u_2\ge 74/73-o(1)$ as $\e\to0$, giving the
inapproximability result for edge-weighted graphs.

Finally, we can get rid of edge weights by choosing a large integer $l$, say
$l=s$, cloning each vertex of $G$ $l$ times, and replacing each edge of weight
$\alpha$ by a pseudo-random bipartite graph of edge density $\alpha$. (The
edge weights in each gadget are real numbers lying between $0$ and $1$.) We
refer the Reader to a survey by Krivelevich and
Sudakov~\cite{krivelevich+sudakov:06} on the properties of pseudo-random
graphs.  Up to a multiplicative error $1+o(1)$ as $l\to\infty$, any cut of the
new graph $G'$ corresponds to a \emph{fractional cut} of $G$, where the
vertices of $G$ may be sliced between the two parts in some ratio and the
value of the cut is defined in the obvious way. However, it is easy to see
that, for an arbitrary edge-weighted (loopless) graph, there is an integer vertex cut
which is at least as good as any fractional cut. Thus $\m(G')=(1+o(1)) l^2
\m(G)$ (and $e(G')=(1+o(1)) l^2 e(G)$) as $l\to\infty$.

The obtained family of (unweighted) graphs $G'$ establishes the lemma.\end{proof}

\brm The weaker result that it is \NP-hard to approximate $3e(G)+\m(G)$ within
a factor $113/112-\delta$ can obtained from the statement of the $17/16$-result of
H\aa stad (without analyzing the structure of his graphs) by observing that
$\m(G)\ge \frac12\, e(G)$ for any $G$ (and doing some easy calculations).
 \medskip

\begin{proof}[Proof of Theorem~\ref{th:hardness}.] Let $l$ be a positive
integer and let $G$ be an arbitrary graph. Define $n=v(G)$, $m=e(G)$, and
$t=\m(G)$.

We construct a new graph $H=H(G,l)$ as follows. Let $V=V(G)$. For each $x\in
V$, introduce a new vertex $x'$. Let $V'=\{x'\mid x\in V\}$ consist of all new
vertices. For each edge $xy\in E(G)$, introduce a set $U_{xy}$ of $l$ new
vertices. Let $U=\cup_{xy\in E(G)} U_{xy}$. The new graph $H$ has $V\cup
V'\cup U$ for the vertex set. Thus the total number of vertices is
$v(H)=2n+lm$. The edges of $H$ are as follows. Let $V$ span the complete
graph. Connect $x$ to $x'$ for each $x\in V$. Put a complete bipartite graph
between $U_{xy}$ and $\{x,y,x',y'\}$ for every $xy\in E(G)$. These are all the
edges (all other pairs of $V(H)$ are non-adjacent). Thus, for example, the
size of $H$ is $e(H)={n\choose 2} + n + 4lm$.

\claim1 $c(H)\ge 3lm+lt$.\medskip 

\bcpf Let $V=X\cup Y$ be a maximum cut of $G$, that is, $e(G[X,Y])=t$. Let
$V'=X'\cup Y'$ be the corresponding partition of $V'$. Let
$X=\{x_1,\dots,x_a\}$ and $Y=\{y_1,\dots,y_b\}$.

Let $P=(V(H),\preceq)$ be the partially ordered set on $V(H)$, where $\prec$ is
the transitive closure of the digraph $D$ that consists of the following
ordered pairs:
 \begin{itemize}
 \item $(x_i,x_{i+1})$ and $(x_i',x_{i+1}')$ for $i\in[a-1]$,
 \item $(y_i,y_{i+1})$ and $(y_i',y_{i+1}')$ for $i\in[b-1]$,
 \item $(x_i,x_i')$ for $i\in[a]$,
 \item $(y_i',y_i)$ for $i\in[b]$,
 \item $(x_a,y_1)$ and $(y_b',x_1')$,
 \item $(x_i,u)$, $(u,x_j)$, and $(u,x_i')$ for $u\in U_{x_ix_j}$
and $x_ix_j\in E(G[X])$ with $i<j$,
 \item $(y_i,u)$, $(y_j',u)$, and $(u,y_j)$ for $u\in U_{y_iy_j}$
and $y_iy_j\in E(G[Y])$ with $i<j$,
 \item $(x,u)$, $(y',u)$, $(u,x')$, and $(u,y)$ for $u\in U_{xy}$ and $xy\in
   E(G[X,Y])$ with $x\in X$.
 \end{itemize} 
 
In other words, we take two chains, namely $x_1\prec\dots\prec x_a\prec
y_1\prec \dots\prec y_b$ and $y_1'\prec\dots\prec y_b'\prec x_1'\prec \dots
\prec x_a'$. We let $x\prec x'$ for $x\in X$ and $y\succ y'$ for $y\in Y$. For
each $x_i\prec x_j$ that are adjacent in $G$, we insert the set $U_{x_ix_j}$
(as an antichain) above $x_i$ but below $x_j$ and $x_i'$. For each $y_i\prec
y_j$ that are adjacent in $G$, we insert the set $U_{y_iy_j}$ (as an
antichain) above $y_i$ and $y_j'$ but below $y_j$. For each $xy\in E(G[X,Y])$
with $x\in X$, we insert the set $U_{xy}$ (as an antichain) above $x$ and $y'$
but below $x'$ and $y$. Figure~\ref{fg:1} shows the placement of the vertices
of $U$ relative to $V\cup V'$. Finally, we add those order relations that are
implied by the above relations.

\begin{figure}[htbp]
\begin{center}

\scalebox{1}{\input{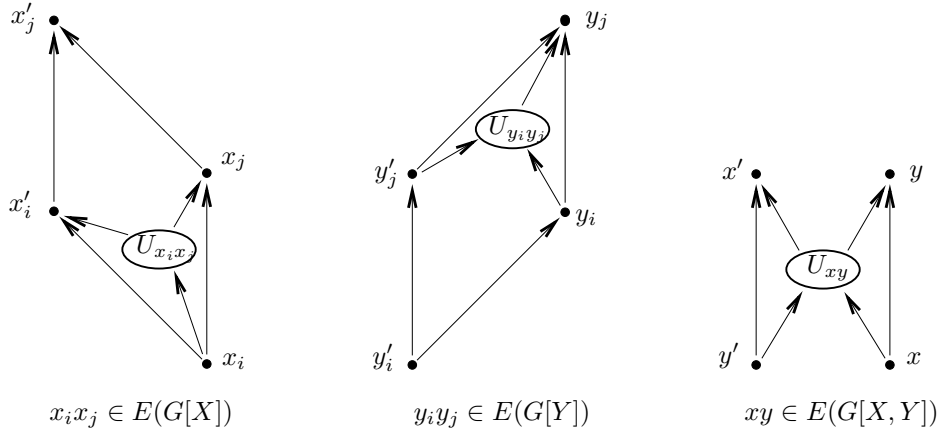}}

\caption{The placement of the set $U$ relative to $V\cup V'$.}
\label{fg:1}
\end{center}
\end{figure}

It is easy to check that $D$ has no oriented cycles and that the obtained
partial order $P$ determines the orientation of every edge of $H$. Strategist
chooses this orientation and answers all Algy's questions accordingly.

The digraph $D$ defined above is not in general the Hasse diagram of the poset
$P$: for example, if $x_ix_{i+1}$ is an edge of $G$, then the relation $x_i\prec
x_{i+1}$ can be determined from $x_i\prec u\prec x_{i+1}$ for some $u\in
U_{x_ix_{i+1}}$. However (and this is the crucial property!) one can routinely
check that every arc of $D$ that connects $V\cup V'$ and $U$ (in either
direction) does belong to the Hasse diagram of $P$, that is, the orientation
of this edge is not determined from the order relation of all other pairs
of~$P$.

Clearly, Algy has to query \emph{every} edge that belongs to the Hasse diagram
of $P$. Thus, Algy has to ask at least $3l$ (resp.\ $4l$) questions per each
edge of $G[X]$ and $G[Y]$ (resp.\ $G[X,Y])$. This shows that $c(H)\ge
3l(m-t)+4lt=3lm+lt$, as required.\ecpf

\claim2 $c(H)\le 3lm+lt+c(K_n)+n$.\medskip

\bcpf Algy finds the orientation of all edges in the clique $H[V]$ by asking
$c(K_n)$ questions. Then he asks about the orientation of every edge $xx'$
with $x\in V$. Let $X$ consist of those $x\in V$ for which we have $(x,
x')$. Let $Y=V\setminus X$.

Take any $xy\in E(G[X])$. Suppose without loss of generality that $x\prec
y$. For each $u\in U_{xy}$, Algy asks about the orientation of the edge
$uy$. Whatever the answer is, it determines the orientation of $ux$ or
$uy'$. Hence, at most $3l$ questions are enough to determine the orientation
of all edges incident to $U_{xy}$. The same applies to the case $xy\in
E(G[Y])$. Finally, Algy asks about all edges incident to $U_{xy}$ where $xy\in
E(G[X,Y])$, posing $4l$ questions per each edge of the cut $\{X,Y\}$. Thus the
total number of questions is at most
 $$
 c(K_n)+n+3l \big(e(G[X]) + e(G[Y])\big) + 4l e(G[X,Y]) = c(K_n)+n +
3le(G) +l e(G[X,Y]),
 $$
 giving the required bound.\ecpf

As it was shown by Ford and Johnson~\cite{ford+johnson:59},
$c(K_n)=(1+o(1)) n\log_2 n$. Thus Claims~1 and~2 show that $c(H)=(1+o(1)) l
(3e(G)+\m(G))$ as $n\to\infty$, if we take $l\gg \ln n$, say $l=n$. (Note
that, by removing isolated vertices from $G$, we can assume that $e(G)\ge
v(G)/2$.)  Since the order of $H$ is bounded by a polynomial in $v(G)$, the
desired inapproximability result for the parameter $c$ follows from
Lemma~\ref{lm:74}.\end{proof}

\section{A General Lower Bound on $c(G)$}\label{lower}

Here is the lower bound on $c(G)$ that implies the approximability result
mentioned at the end of the Introduction.

\bth{lower} There is a constant $C>0$ such that any graph $G$
satisfies 
 \beq{lower}
 c(G)\ge \frac{e(G)\log_2(v(G))}{Cv(G)}.
 \eeq
 \end{theorem}
 \begin{proof} Fix a sufficiently large $C$.
 Let $G$ be an arbitrary graph of order $n$ and size $m$.

Clearly, it is enough to prove the theorem under the assumption that $G$ has
no isolated vertices. Indeed, if we remove isolated vertices, then $c(G)$
remains the same while the right-hand side of~\req{lower} can only
increase. 

We have $c(G)\ge n/2$ because, for every vertex $x$ of $G$, we have
to query at least one edge incident to $x$. It follows that~\req{lower} holds
unless
 \beq{m}
 m> \frac{Cn^2}{2\log_2 n}.
 \eeq

So suppose that~\req{m} holds.  The average degree of $G$ is $2m/n$. If we
remove a vertex whose degree is less than $m/n$, then the average degree of
$G$ goes up. By iteratively repeating this step, we can find a non-empty set
$X\subseteq V(G)$ such that the induced subgraph $H=G[X]$ has minimum degree
at least $d=\ceil{m/n}$.

The graph $H$ contains at least $|X|\, d!\ge (d+1)!$ directed paths $P$ of
length $d$: there are $|X|$ choices for the first vertex and, inductively for
$i=2,\dots, d+1$, at least $d-i+2$ choices for the $i$-th vertex. For every
choice of $P$ choose an acyclic orientation of the whole graph $G$ compatible
with the orientation of $P$. Clearly, each orientation of $G$ can appear this
way for at most ${n\choose d+1}\le 2^n$ different directed $d$-paths
$P$. Hence, $a(G)$, the number of acyclic orientations of $G$, is at least
$(d+1)!/2^n$. The usual information theoretic lower bound (see, for example,
Aigner~\cite[Page~24]{aigner:cs}) implies that
 $$
 c(G)\ge \log_2(a(G))\ge \log_2\left(\frac{(d+1)!}{2^n}\right).
 $$
 If $C$ is large, then also $n$ is large by~\req{m} and because 
$m\le {n\choose 2}$. Again
by~\req{m}, we have 
 \beq{d}
 d\ge \frac mn> \frac{Cn}{2\log_2n}\quad \mbox{and}\quad \log_2d >\frac{\log_2n}2,
 \eeq
 so $d$ is forced to be large too.
By Stirling's formula, $\log_2((d+1)!)>0.9\, d\log_2 d$. We have
by~\req{d} that, for example, $d\log_2
d> (C n/(2\log_2n))\times (\log_2 n)/2> 2 n$. Thus
 $$
 \log_2\left(\frac{(d+1)!}{2^n}\right) > 0.9\,d\log_2d - n > 0.4\,
 d\log_2d\ge 0.4\times \frac mn\times \frac{\log_2 n}2,
 $$ 
giving the required.\end{proof}

\brm The inequality in~\req{lower} is sharp (up to an $O(1)$-factor) when $G$
is the complete graph $K_n$ or, more generally, when $G$ is a typical graph in
$\C G_{n,p}$ with constant edge-probability $p>0$ by the result of Alon and
Tuza~\cite{alon+tuza:95}.\medskip

\section*{Acknowledgments}

The author thanks Alan Frieze, Tao Jiang, and Oleg Verbitsky for helpful
comments.


\begin{thebibliography}{10}

\bibitem{aigner:cs}
M.~Aigner.
\newblock {\em Combinatorial search}.
\newblock Wiley-Teubner Series in Computer Science. John Wiley \& Sons Ltd.,
  Chichester, 1988.

\bibitem{aigner+triesch+tuza:95}
M.~Aigner, E.~Triesch, and Z.~Tuza.
\newblock Searching for acyclic orientations of graphs.
\newblock {\em Discrete Math.}, 144:3--10, 1995.

\bibitem{alon+tuza:95}
N.~Alon and Z.~Tuza.
\newblock The acyclic orientation game on random graphs.
\newblock {\em Random Struct.\ Algorithms}, 6:261--268, 1995.

\bibitem{bollobas:mgt}
B.~{Bollob\'as}.
\newblock {\em Modern Graph Theory}.
\newblock Springer-Verlag, Berlin, 1998.

\bibitem{bollobas+rosenfeld:81}
B.~Bollob{\'a}s and M.~Rosenfeld.
\newblock Sorting in one round.
\newblock {\em Israel J.\ Math.}, 38:154--160, 1981.

\bibitem{chernoff:52}
H.~Chernoff.
\newblock A measure of asymptotic efficiency for tests of a hypothesis based on
  the sum of observations.
\newblock {\em Ann.\ Math.\ Statistics}, 23:493--507, 1952.

\bibitem{diestel:gt}
R.~Diestel.
\newblock {\em Graph Theory}.
\newblock Springer, Berlin, 3rd edition, 2006.

\bibitem{erdos:67a}
P.~Erd{\H{o}}s.
\newblock Some recent results on extremal problems in graph theory. {R}esults.
\newblock In {\em Theory of Graphs (Internat. Sympos., Rome, 1966)}, pages
  117--123 (English); pp. 124--130 (French). Gordon and Breach, New York, 1967.

\bibitem{ford+johnson:59}
L.~R. Ford and S.~M. Johnson.
\newblock A tournament problem.
\newblock {\em Amer.\ Math.\ Monthly}, 66:387--389, 1959.

\bibitem{gasarch+golub+kruskal:03}
W.~Gasarch, E.~Golub, and C.~Kruskal.
\newblock Constant time parallel sorting: an empirical view.
\newblock {\em J.\ Computer Syst.\ Sci.}, 67:63--91, 2003.

\bibitem{haggkvist+hell:81}
R.~H{\"a}ggkvist and P.~Hell.
\newblock Parallel sorting with constant time for comparisons.
\newblock {\em {SIAM} J.\ Computing}, 10:465--472, 1981.

\bibitem{hastad:97}
J.~H{\aa}stad.
\newblock Some optimal inapproximability results.
\newblock In {\em Proceedings of the 29th ACM Symposioum on Theory of Computing
  (El Paso, TX, 1997)}, pages 1--10. ACM, New York, 1999.

\bibitem{hastad:01}
J.~H{\aa}stad.
\newblock Some optimal inapproximability results.
\newblock {\em J.\ ACM}, 48:798--859, 2001.

\bibitem{jiang:pc}
T.~Jiang, 2008.
\newblock Personal communication.

\bibitem{knuth:acp3}
D.~E. Knuth.
\newblock {\em The art of computer programming. {V}olume 3: Sorting and
  searching}.
\newblock Addison-Wesley Publishing Co., 1973.

\bibitem{krivelevich+sudakov:06}
M.~Krivelevich and B.~Sudakov.
\newblock Pseudo-random graphs.
\newblock In {\em More sets, graphs and numbers}, volume~15 of {\em Bolyai Soc.
  Math. Stud.}, pages 199--262. Springer, Berlin, 2006.

\bibitem{manber+tompa:81}
U.~Manber and M.~Tompa.
\newblock The effect of number of {H}amiltonian paths on the complexity of a
  vertex-coloring problem.
\newblock In {\em Proceedings of the 22nd Annual Symposium on Foundations of
  Computer Science (Nashville, TN)}, pages 220--227. IEEE Comput. Soc. Press,
  1981.

\bibitem{manber+tompa:84}
U.~Manber and M.~Tompa.
\newblock The effect of number of {H}amiltonian paths on the complexity of a
  vertex-coloring problem.
\newblock {\em {SIAM} J.\ Computing}, 13:109--115, 1984.

\bibitem{mantel:07}
W.~Mantel.
\newblock Problem 28.
\newblock {\em Winkundige Opgaven}, 10:60--61, 1907.

\bibitem{moon+moser:62}
J.~W. Moon and L.~Moser.
\newblock On a problem of {Tur\'an}.
\newblock {\em Publ.\ Math.\ Inst.\ Hungar.\ Acad. Sci.}, 7:283--287, 1962.

\bibitem{peczarski:02}
M.~Peczarski.
\newblock Sorting 13 elements requires 34 comparisons.
\newblock In {\em Proceedings of the 10th Annual European Symposium on
  Algorithms}, volume 2461 of {\em Lecture Notes in Comput. Sci.}, pages
  785--794. Springer, Berlin, 2002.

\bibitem{peczarski:04}
M.~Peczarski.
\newblock New results in minimum-comparison sorting.
\newblock {\em Algorithmica}, 40:133--145, 2004.

\bibitem{ruzsa+szemeredi:78}
I.~Z. Ruzsa and E.~{Szemer\'edi}.
\newblock Triple systems with no six points carrying three triangles.
\newblock In A.~Hajnal and V.~{S\'os}, editors, {\em Combinatorics {II}}, pages
  939--945. North Holland, Amsterdam, 1978.

\bibitem{simonovits:68}
M.~Simonovits.
\newblock A method for solving extremal problems in graph theory, stability
  problems.
\newblock In {\em Theory of Graphs (Proc. Colloq., Tihany, 1966)}, pages
  279--319. Academic Press, 1968.

\bibitem{szemeredi:76}
E.~Szemer{\'e}di.
\newblock Regular partitions of graphs.
\newblock In {\em Proc.\ Colloq. Int.\ CNRS}, pages 309--401. Paris, 1976.

\bibitem{trevisan+sorkin+sudan+williamson:96}
L.~Trevisan, G.~B. Sorkin, M.~Sudan, and D.~P. Williamson.
\newblock Gadgets, approximation, and linear programming.
\newblock In {\em Proceedings of the 37th Annual Symposium on Foundations of
  Computer Science (Burlington, VT, 1996)}, pages 617--626. IEEE Comput. Soc.
  Press, Los Alamitos, CA, 1996.

\bibitem{trevisan+sorkin+sudan+williamson:00}
L.~Trevisan, G.~B. Sorkin, M.~Sudan, and D.~P. Williamson.
\newblock Gadgets, approximation, and linear programming.
\newblock {\em {SIAM} J.\ Computing}, 29:2074--2097, 2000.

\bibitem{turan:41}
P.~{Tur\'an}.
\newblock On an extremal problem in graph theory (in {Hungarian}).
\newblock {\em Mat.\ Fiz.\ Lapok}, 48:436--452, 1941.

\bibitem{tuza:01}
Z.~Tuza.
\newblock Unsolved combinatorial problems.
\newblock BRICS Lecture Series, LS-01-1, 2001
(available from \texttt{http://www.brics.dk/publications/}).

\bibitem{valiant:75}
L.~G. Valiant.
\newblock Parallelism in comparison problems.
\newblock {\em {SIAM} J.\ Computing}, 4:348--355, 1975.

\end{thebibliography}
\end{document}